\documentclass[letterpaper, 10 pt, conference]{ieeeconf}

\usepackage{dsfont}
\usepackage[T1]{fontenc}
\usepackage{graphicx}
\usepackage{soul}
\usepackage{amsfonts}
\usepackage{amsmath}
\usepackage{setspace}
\usepackage{amssymb}
\usepackage{booktabs,graphicx}
\usepackage[noadjust]{cite}
\usepackage{bbm}
\usepackage{hyperref}
\usepackage{epstopdf}
\usepackage{wasysym}
\usepackage{textcomp}
\usepackage{xcolor}

\usepackage{cite}

\IEEEoverridecommandlockouts

\usepackage{relsize}

\DeclareMathAlphabet{\mymathbb}{U}{BOONDOX-ds}{m}{n}

 \newtheorem{thrm}{Theorem}

\newtheorem{lem}{Lemma}

\newtheorem{prop}{Proposition}

\newtheorem{defn}{Definition}

\newtheorem{rmk}{Remark}

 \allowdisplaybreaks  

\renewcommand{\baselinestretch}{1}

\newcommand{\yo}[1]{{\color{black}{#1}}}
\newcommand{\yl}[1]{{\color{black}{#1}}}

\usepackage[prependcaption,colorinlistoftodos]{todonotes}

\newcommand{\kl}{\textcolor{black}}
\newcommand{\kkl}{\textcolor{black}}

\newcommand{\khl}{\textcolor{black}}
\newcommand{\nkl}{\textcolor{black}}
\newcommand{\Nkl}{\textcolor{black}}

\newcommand{\il}{\textcolor{black}}
\newcommand{\ill}{\textcolor{black}}
\newcommand{\illl}{\textcolor{black}}
\newcommand{\ic}{\textcolor{black}}

\newcommand{\yy}[1]{{\color{black}{#1}}}
\newcommand{\cj}{\textcolor{black}}
\newcommand{\jo}[1]{{\color{black}{#1}}}
\newcommand{\ym}[1]{{\color{black}{#1}}}
\newcommand{\lk}{\textcolor{black}}
\newcommand{\li}{\textcolor{black}}
\newcommand{\lic}{\textcolor{black}}
\newcommand{\licc}{\textcolor{black}}
\newcommand{\lli}{\textcolor{black}}

\newcommand{\yr}[1]{{\color{black}{#1}}}
\newcommand{\khlb}{\textcolor{black}}
\newcommand{\ir}[1]{{\color{black}{#1}}}
\newcommand{\irr}[1]{{\color{black}{#1}}}
\newcommand{\icr}[1]{{\color{black}{#1}}}

\title{\jo{On the synchronization of the Kuramoto-type model of oscillators with lossy  couplings}}

\author{
	Yemi Ojo,  Khaled Laib and Ioannis Lestas   
	\thanks{This work was supported by ERC starting grant 679774.}
	\thanks{The authors are with the Department of Engineering, University of Cambridge, Trumpington Street, Cambridge, CB2 1PZ, United Kingdom. Emails: \{yo259,  kl507, icl20\}@cam.ac.uk}
}

\begin{document}

	\maketitle
	\thispagestyle{empty}
	\pagestyle{empty}

	\begin{abstract}		

		\jo{We consider the problem of  synchronization of coupled oscillators in a Kuramoto-type model with lossy couplings. 				
			Kuramoto models have been used to gain insight on the stability of power networks which are usually nonlinear and \cj{involve large scale interconnections.} 
\li{Such models commonly assume lossless couplings and  Lyapunov functions} 		
		 \li{have predominantly been} employed to prove stability.
		However,   coupling conductances   can impact synchronization. We  therefore consider a more advanced Kuramoto model that  \cj{includes \lli{coupling conductances},}
			\yr{and is characterized by nonhomogeneous coupling weights and noncomplete coupling graphs.}		
		Lyapunov analysis once \cj{such} \li{coupling} conductances \yr{and aforementioned properties} are included \cj{becomes} nontrivial  \cj{and more conventional energy-like  Lyapunov functions are not applicable or are conservative.}
		\yr{Small-signal analysis has been performed for such models, but due to \ir{the fact that we have convergence to a manifold, stability analysis via a linearization}
is on its own inconclusive for the nonlinear \ir{model}.}
	\khlb{In this paper, we provide a formal derivation using centre manifold theory that if a particular condition on the equilibrium point associated with the  coupling conductances and susceptances  holds, then
		\cj{the synchronization manifold for the \ir{nonlinear} system considered is asymptotically stable.}}
		Our analysis is demonstrated with simulations.}
		
	\end{abstract}
	
	
	\section{Introduction}%

\jo{\li{Stability analysis of power networks is a problem that has received a considerable amount of attention.
 Power networks}
are generally a complex interconnection of
	\cj{subsystems that include also nonlinear dynamics.}
	 These include controllable inverter-based microgrids which have been identified to offer good prospects for the integration of renewable energy sources in future grids.
Considering the nonlinear and complex dynamics associated with power systems, simple models which are valid on slow timescales have been used to gain insight into the stability of their synchronized motion. \li{Such a model}
\lli{is the Kuramoto} model of \lli{coupled oscillators~
	\cite{kuramoto2003chemical}.}
This was initially used to   study chemical oscillators \cite{kuramoto2003chemical},  and has since been employed in investigating the stability of power networks \cite{dorfler2012synchronization,simpson2013,  schiffer2013synchronization} as well as consensus  protocols of multi-agent systems \cite{olfati2007consensus, lin2007state}.
Generally, Kuramoto models can be classified into the first-order type \cite{jadbabaie2004stability, chopra2009exponential, dorfler2011critical, dorfler2012exploring,simpson2013} and second-order type \cite{dorfler2012synchronization, schiffer2013synchronization}.}


The stability problem of
\jo{coupled oscillators in  Kuramoto models}  is a particular form of  \ic{a} synchronization problem.
This is characterized by the synchronization of the frequencies of \jo{coupled oscillators} to a common constant value. For power networks, this implies that the frequency of the  generators (inverters) converge to a synchronous value.
 The synchronization of \jo{coupled Kuramoto oscillators has been widely reported in the literature, e.g.
\cite{jadbabaie2004stability, chopra2009exponential, dorfler2011critical, dorfler2012exploring} and with particular  application to   inverter-based microgrids  \cite{simpson2013,schiffer2013synchronization}.}

\jo{
\li{Kuramoto-type models} have been used in the literature to provide intuition on the \li{behaviour} 
of power networks \li{at} 
  slow \li{timescales, by providing analytical results that hold in general network topologies.}
  \li{Such models have commonly been used to describe} 
  oscillators interconnected via lossless couplings \cite{jadbabaie2004stability, chopra2009exponential, dorfler2011critical, dorfler2012exploring,simpson2013}, and thus the analysis of these Kuramoto-type \li{models} has been predominantly performed via Lyapunov functions.
	 However,   the line conductances (or resistances) in e.g. inverter-based grids are important as they can inhibit   synchronization.
	 \yr{Also, these power networks are usually characterized by \ir{nonhomogeneous}
coupling weights and \ir{noncomplete}
\icr{coupling graphs.}}
	 \lk{Hence, here  we consider a more advanced Kuramoto model that   includes coupling conductances}
	 \yr{together with heterogeneous coupling weights and noncomplete coupling graphs, in contrast to those in e.g. \cite{chopra2009exponential,jadbabaie2004stability,ha2014large}.}
	 Lyapunov analysis once \cj{such} conductances \yr{and aforementioned coupling properties} are included \cj{becomes nontrivial as more conventional energy-like Lyapunov functions are not applicable or are conservative, \yr{as reported in}  \cite{dorfler2012synchronization}.}
	 \yr{Small-signal analysis for this case exists in the literature, e.g \cite{schiffer2013synchronization}; however, since the system converges to \ir{a manifold
a small-signal} analysis is on its own  inconclusive for  the nonlinear system \cite{khalil2014}.	
 }
}


	\khlb{\yr{In this paper, we study the stability of the system} 
		using centre manifold theory \cite{khalil2014, wiggins2003introduction, jouini2021frequency}. \yr{Centre manifold theory  allows to study the stability of nonlinear systems  when linearization fails \cite{khalil2014}.} 
		Using this theory, we provide a formal derivation  that if a particular \ir{condition
involving the} coupling conductances and susceptances of the lossy Kuramoto-type model holds \ir{at an equilibrium point}, the synchronization manifold for the system considered is asymptotically stable.}	
%
%
%
%
%
%
%
\yr{
\ir{A numerical example is also provided} to demonstrate the \ir{results} of our analysis.}

\khlb{The remainder of the paper is organized as follows.
\jo{Preliminaries and \li{the} problem setup are  given in  section \ref{nota}.
Section \ref{main} presents our main results.
\li{A simulation} example is given in section~\ref{simureslt} and \li{conclusions are} drawn in section~\ref{conclu}.}}

	\section{\jo{Preliminaries and Problem Setup}}
	\label{nota}
	Let $\mathbb{R}_{\geq0}=\{x\in\mathbb{R}|x\geq0\}$,  $\mathbb{R}_{>0}=\{x\in\mathbb{R}|x>0\}$.
	\yo{We denote by}
	 $\mathbf{1}_n$ $(\mathbf{0}_n)$  the $n$-dimensional column	vector of ones (zeros), and  $\mathbf{I}_{n}$ is the	identity matrix of size $n$, and $\mathbf{j}$ a complex number.
	Given an $n$-tuple $(x_1,\ldots,x_n)$,
	$x=\text{col}(x_i) \in \mathbb{R}^{n}$ \li{is a column} vector, $\text{diag}(x_i)\in\mathbb{R}^{n\times n}$ \li{is an $n\times n$ diagonal matrix},
	\yo{and $\|x\|$ denotes \li{the 2-norm of $x$ and $\|x\|_{\infty}$ its} infinity norm.}
	\yo{Given $A\in\mathbb{R}^{n\times m}$, let $A^{\top}$ denote its transpose.} 
	\li{Let 
$\mathbb{T}^{1}:=[-\pi, \pi]$,}
	 a phase angle $\delta_i$ is a point $\delta_i\in\mathbb{T}^{1}$, \yo{and}
	an arc with length $\gamma$ is \kl{a} connected subset of $\mathbb{T}^{1}$.
	
	\subsection{Graph theory}
	\label{nota-graph}
	We denote a directed  graph by $\mathcal{G}=(\mathcal{V},\mathcal{E})$, where $\mathcal{V}=\{1, 2, \ldots, N\},\,N=|\mathcal{V}|$ is the set \kkl{of vertices  (nodes)}, and   $\mathcal{E}\subseteq \mathcal{V}\times \mathcal{V}$ is the set of directed edges with $E=|\mathcal{E}|$.
	Node $j\in\mathcal{V}$ is a neighbour of a node $i\in\mathcal{V}$ if  $(i,j)\in \mathcal{E}$, and the set of neighbours of a node $i\in\mathcal{V}$ is denoted by $\mathcal{N}_i$.
	For each edge $(i,j)\in \mathcal{E}$,
	a number $z\in\{1, \ldots,E\}$ and an arbitrary direction are assigned and the entries of the incidence  matrix $\mathcal{B} \in\mathbb{R}^{N\times E}$ are defined as  $\mathcal{B}_{iz}=1$ if node $i$ is the source of edge $z$ and $\mathcal{B}_{iz}=-1$ if node $i$ is the sink of edge $z$, with all other elements being zero.
	\yo{Let the graph $\mathcal{G}$ be connected}, i.e.  for all $i,j\in\mathcal{V}$, there exists a path from $i$ to $j$.
	Hence, \yl{$\text{span}(\mathbf{1}_N)\in\ker(\mathcal{B}^{\top})$.}
	Let ${\mathbf{\hat{B}}} \in\mathbb{R}^{N\times E}$ be the element-wise absolute \ill{value of $\mathcal{B}$.}

	\subsection{Problem setup}\label{model}
	\jo{We consider in this paper a Kuramoto-type model that takes coupling conductances into account.
		This is briefly described as
		\lli{follows (\cite{dorfler2012synchronization} can be consulted for an analogous derivation from first principles)}.
		 	Let the interconnection of the oscillators be described by the connected graph $\mathcal{G}$.	Let each oscillator and load be connected at each node $i\in\mathcal{V}$, and the coupling between every node pair $(i, j) \in \mathcal{E}$ be characterized by an admittance $Y_{ij}=G_{ij}+\mathbf{j}B_{ij}$, 	where
	$G_{ij}, B_{ij} \in\mathbb{R}_{>{0}}$ are  the coupling conductance (losses) and susceptance respectively.
	Note that in the case of nodes to which only loads are connected, these can be eliminated using Kron-reduction \cite{kundur94,dorfler2012kron} leading to a lower dimensional representation of the original model where each  node has an oscillator connected to it.
	\lli{We define}}
	\begin{equation}  \label{Pn}
		\hspace{-0.62cm}	\begin {matrix}		
		&  \varpi_i=P^d_i-P_{L_i} -  \sum_{ j\in\mathcal{N}_i} G_{ij}V^2_i,
		\\&a_{ij}=V_iV_jG_{ij},~~ b_{ij}=V_iV_jB_{ij}, ~~\delta_{ij}(t)=\delta_i(t)-\delta_j(t),
		\vspace{-0.125cm}
	\end{matrix}
\end{equation}	
\jo{where $V_i, V_j \in\mathbb{R}_{>0}$ are the internal voltages of the oscillators at nodes $i, j$; $a_{ij}$, $b_{ij} \in\mathbb{R}_{>0}$ are respectively the lossy and lossless coupling  weights; $P^d_{i}, P_{L_i}\in\mathbb{R}_{>0}$ are respectively the desired power injection and the constant  power load at a node $i$; and $\varpi_i$ denotes the \li{effective power injection at node $i$ which is often referred to in Kuramoto models as the natural frequency \cite{dorfler2012synchronization} at that node. Variables} 
$\delta_i(t), \delta_j(t)\in \mathbb{T}^1$ \li{denote} the phase angle of the oscillators at node $i$ and $j$, \li{respectively}; \li{$\delta_{ij}(t)$ is} the phase difference between the  node pair $(i, j)\in\mathcal{E}$.	
	 Therefore, the dynamics of the Kuramoto model considered in this paper describing the interconnection of oscillators via lossy couplings \li{are} given by}
\begin{equation}
	\label{slow2}
	\hspace{-0.2cm}\dot\delta_i(t)=k_{pi}\varpi_i-k_{pi} \hspace{-0.2cm}  \displaystyle\sum_{ j\in\mathcal{N}_i}(b_{ij}\sin(\delta_{ij}(t))-a_{ij}\cos(\delta_{ij}(t)))\vspace{-0.125cm} 
\end{equation}
 \jo{for all $i\in\mathcal{V}$, with gain $k_{pi}\in\mathbb{R}_{>0}$.
 	\yr{Note that the coupling weights $a_{ij}, b_{ij}$ in  \eqref{slow2}  \ir{are allowed to be} nonhomogeneous, i.e. the weights are not \ir{necessarily} \icr{the same} 
 for \icr{every edge} $(i,j) \in\mathcal{E}$ \ir{(as assumed in e.g.  \cite{chopra2009exponential,jadbabaie2004stability, ha2014large}).} Also, \ir{each node is not required to be connected to all other nodes, i.e. each node} 
 $i\in\mathcal{V}$ connects to its neighbouring nodes $j\in\mathcal{N}_i$ }
 \ir{where $\mathcal{N}_i\subset\mathcal{V}$}.
 	\yr{Furthermore,} the term   $a_{ij}$ being associated with the coupling conductance $G_{ij}$ in \eqref{Pn} allows the Kuramoto-type model \eqref{slow2} to take into account the coupling losses.
 	 Model \eqref{slow2} \li{is also referred} to as a first-order Kuramoto-type model.

 Note that  $a_{ij}$, $b_{ij}$ are respectively associated with the coupling conductance $G_{ij}$ and susceptance $B_{ij}$.
 Thus the term $b_{ij}\sin(\delta_{ij}(t))$ usually enables  synchronization given its negative feedback role in \eqref{slow2} while  $a_{ij}\cos(\delta_{ij}(t))$ can inhibit  synchronization  due to its positive feedback role in \eqref{slow2}.}

\begin{rmk}
	\yy{System \eqref{slow2} is a simplified model relevant for studying inverter-based power networks at slower timescales \cite{simpson2013,schiffer2014}. This can be related  \lic{to such networks } as follows.
	Let $V_i, V_j$  denote the  inverter \li{voltages at the corresponding buses}
		and the couplings represented by power lines with conductances  $G_{ij}$ and susceptances $B_{ij}$.
\li{The phase angle $\delta_i(t)$ is associated  with the 
\li{angle at each bus},
and its derivative is defined}
as the deviation between \li{the local} frequency $\omega_i(t)\in\mathbb{R}_{>0}$ and a common frequency $\omega_{0}\in\mathbb{R}_{>0}$ \cite{schiffer2014},
 \lic{i.e.,} $\dot{\delta}_i(t)=\omega_i(t)-\omega_{0}:=u_i^{\omega}(t),\,\forall i\in\mathcal{V}$. Suppose the control input $u_i^{\omega}(t)$  is generated via a droop law $u_i^{\omega}(t)=-k_{pi}(P_i(t)-P^d_i)$ \cite{schiffer2014}
where the gain $k_{p_i}\in\mathbb{R}_{>0}$ is  now the droop gain,	and
	$P_i(t)\in\mathbb{R}_{>0}$ is the overall active power injection
	at a node 	$i\in\mathcal{V}$ given by \cite{kundur94,schiffer2014, SCHIFFER2016135}
	$P_i(t)=  P_{L_i}+\sum_{ j\in\mathcal{N}_i}
	[G_{ij}V^2_i-G_{ij}V_i  V_j\cos(\delta_i(t)-\delta_j(t))
	+B_{ij}V_iV_j\sin(\delta_i(t)-\delta_j(t))]$. Then \cj{setting} $\dot{\delta}_i(t):=u_i^{\omega}(t)$ \cj{leads to} the Kuramoto-type model in \eqref{slow2}.}
\end{rmk}


\jo{In the subsequent analysis,  we  use \eqref{slow2} to investigate  the synchronization of the coupled oscillators.}


We note that in the \il{remainder} of this paper,    the  \il{arguments $t$
	of all} signals are omitted to simplify \il{the} notation.

\section{Main Results}\label{main}


\khlb{We present in this section our  main results. In particular, using centre manifold theory we show that if a particular
	condition associated with the  line conductances and \icr{susceptances is satisfied} by an equilibrium point, then local asymptotic stability of the corresponding synchronization manifold holds.}
	

	\yr{To this end, we recall the definitions for frequency synchronization (Definition \ref{defnsyn}) \li{and the synchronization (Definition \ref{defn_feasibility}) set that will be considered} in our analysis.}

	\begin{defn}[Frequency synchronization]\label{defnsyn}
	\ym{Let $\delta=\text{col}(\delta_i)\in\mathbb{T}^n$}.
	{The oscillators \yo{in~\eqref{slow2}} \yl{have achieved} synchronization if there exists    a \yo{constant} frequency deviation $\dot\delta_{\text{syn}}$ 
such that
\yl{$\dot\delta(t)=\mathbf{1}_N\kl{\dot \delta_{\text{syn}}}$}.}
\end{defn}

\begin{defn}[Synchronization set \cite{dorfler2012synchronization}]\label{defn_feasibility}
{For some {$\gamma\in[0,\pi/2-\psi_{\max})$}, $\psi_{\max}=\max_{(i,j)\in\mathcal{E}} \arctan(G_{ij}/B_{ij})$, we define the set
\mbox{$\mathcal{C}(\gamma){:=}\{\delta:\|\mathcal{B}^{\top}\delta\|_{\infty}\leq\gamma\}\subset\mathbb{T}^n$},
{i.e. this is the}  closed set of angles \mbox{$\delta=\text{col}(\delta_1,\ldots,\delta_N)$}  with any neighbouring angle pair \mbox{$\delta_{ij}, \,\forall (i,j)\in\mathcal{E}$}, no further than $\gamma$ apart.}
\end{defn}

\khlb{Before presenting our main result on the synchronization of system \eqref{slow2}, we illustrate the limitations of 
existing approaches when considering lossy interconnections.} For lossless \cj{interconnections}, \icr{i.e. $a_{ij}=0$, the} only trigonometric term that appears in \eqref{slow2} is $b_{ij}\sin(\delta_{ij})$.
	This term  has the property $b_{ij}\sin(\delta_{ij})=-b_{ij}\sin(\delta_{ji})$   which allows to find Lyapunov functions and  simplifies significantly the analysis.
	However,  this property is not  satisfied by the other term  $a_{ij}\cos(\delta_{ij})$ which is   associated with $G_{ij}$.
	\yr{
\ir{The nonhomogeneous parameters and noncomplete interconnection graph 
in \eqref{slow2} further complicates the analysis, and Lyapunov analysis once 
coupling} conductances $G_{ij}$ are included \ir{can be}  conservative \cite{schiffer2013synchronization}.}

\khlb{	\yr{\ir{Theorem \ref{thrmfreq} is therefore based on a different approach, whereby 
the centre manifold theorem is used to analyze the stability properties of the nonlinear system. In particular, we show} that the  synchronization  manifold for system \eqref{slow2} is locally asymptotically stable if a particular \ir{condition associated} with the  line conductances and susceptances \ir{is satisfied at an equilibrium \icr{point.}}}} 

	\begin{thrm}[Frequency synchronization]\label{thrmfreq}
		Consider  system \yr{\eqref{slow2}, together with}
		\li{\eqref{Pn}}.	
		Assume there exists an equilibrium point  \jo{$\delta^*$  that satisfies  $\delta^*\in\mathcal{C}(\gamma)$} for some $\gamma\in[0,\pi/2-\psi_{\max})$ \jo{with $\psi_{\max}=\max_{(i,j)\in\mathcal{E}} \arctan(G_{ij}/B_{ij})$,} \ic{and consider the manifold}\vspace{-0.125cm} 	
		\begin{equation}\label{synch-manifold}
			\yl{	\mathcal{Y}_{\delta^*}=\left\lbrace  \delta^*+\delta_0\mathbf{1}_N  \Nkl{\ | \ \delta_0\in \mathbb{T}^1 
				}\right\rbrace .}\vspace{-0.125cm}
		\end{equation}
		\jo{Then the} synchronization manifold  \yl{$\mathcal{Y}_{\delta^*}$} is locally asymptotically stable. \licc{Furthermore,} 
		there exists a neighbourhood $\Xi$ of \yl{$\mathcal{Y}_{\delta^*}$} such that for every $\delta(0)\in \Xi$,
		\yl{$\lim_{t\to\infty}\|\delta(t)-\delta^{\dagger}\|\to0$
			\ill{\illl{for some} $\delta^{\dagger}\in\mathcal{Y}_{\delta^*}$}}.
	\end{thrm}
	
	The proof can be found in \ym{Appendix \ref{proofsyn-a}.}

	\begin{rmk}
		\jo{Using \cj{the} centre manifold theorem, Theorem \ref{thrmfreq} \lli{shows that} the synchronization manifold for the lossy case \eqref{slow2} is asymptotically stable for  an equilibrium point $\delta^*$   that satisfies a condition associated with the  line conductances and susceptances, i.e. for
			$\delta^*\in\mathcal{C}(\gamma)$ for some $ \gamma\in[0,\pi/2-\psi_{\max})$ with $\psi_{\max}=\max_{(i,j)\in\mathcal{E}} \arctan(G_{ij}/B_{ij})$.
			This  is \lli{a condition} analogous to \lli{that} for the lossless case e.g. \cite{jadbabaie2004stability,chopra2009exponential,dorfler2011critical,dorfler2013synchronization,dorfler2012exploring} where Lyapunov analysis is used to show asymptotic stability	for  an equilibrium point 		$\delta^*\in\mathcal{C}(\gamma)$ with $ \gamma\in[0,\pi/2)$.}
	\end{rmk}

	\begin{rmk}
\ir{We would like to note that system \eqref{slow2} has the following equivalent representation (stated as Lemma \ref{altmodel} in Appendix \ref{proofaltmodel}).
	\vspace{-0.125cm}
	\begin{equation} \hspace{-0.4cm} \label{eq:eqRep}
		\dot{\delta}_i=k_{pi}\varpi_i- k_{pi} \hspace{-0.1cm}\sum_{ j\in\mathcal{N}_i}
		\sqrt{a^2_{ij}+b^2_{ij}}\,\sin(\delta_{ij}-\psi_{ij}), \hspace{2mm} \forall i\in\mathcal{V} \vspace{-0.125cm}	\end{equation}
	where \mbox{$\psi_{ij}=\arctan(a_{ij}/b_{ij})=\arctan(G_{ij}/B_{ij})$}. Therefore $\psi_{ij}$ correspond to phase shifts and $\psi_{\max}$ used in Theorem \ref{thrmfreq} is the maximum  of these  phase shifts.}
\yr{Since $G_{ij}, B_{ij}>0$ (or $a_{ij}, b_{ij}>0$),
	 practical values\yr{\footnote{In the case of lossless couplings $\psi_{ij}=0$ since  $G_{ij}=0$ (i.e. $a_{ij}=0$).}} of $\psi_{ij}$ are in the range $0<\psi_{ij}<\pi/2$.
	The implication of the phase shifts $\psi_{ij}$ is discussed in Remark \ref{remark-psi} below.}
	\end{rmk}
	
	\begin{rmk}\label{remark-psi}
		It is a well-known fact that \ir{asymptotic stability of the synchronization manifold is guaranteed} in \yr{lossless networks} when the magnitude of the	phase differences $|\delta_{ij}|$ \ir{at equilibrium} \yr{do not}  exceed  $\pi/2$  \cite{kundur94,simpson2013,dorfler2013synchronization}.
		However, \yr{the  presence of the
			phase shifts $\psi_{ij}$  in  \eqref{slow2}, \ir{as revealed by its equivalent representation \eqref{eq:eqRep},  
}} imposes additional {constraints} on the  values that $|\delta_{ij}|$  can take to guarantee frequency synchronization in \li{system~\eqref{slow2}.} 		
		\yr{This provides insight into why power networks such as small-scale inverter-based  microgrids can often encounter stability issues.
			For these networks, the resistances of their couplings (interconnecting power lines) can be significant
			and thus  $\psi_{ij}$ can take large values close to $\pi/2$.  This  \icr{results in} 
high  resistance to the flow  necessary for synchronization since the conductance $G_{ij}$  responsible for inhibiting synchronization is large, thus explaining why 
			these networks can often encounter stability issues.
			
\ir{This} is the problem 
			\ir{design approaches} 
such as the
			use of  inductors to couple~inverters (oscillators) to the   nodes (buses) \cite{pogaku2007,yemijeremy2019,yemiisgt2021} and the implementation~of virtual impedances in inverters \cite{guerreroLuis2005,he2013islanding,ysunGuerrero2017}, aim to solve. Being chosen to be predominantly inductive, these coupling inductors and virtual impedances allow to \icr{increase} 
the line susceptances $B_{ij}$ (which supports synchronization), thereby rendering  $\psi_{ij}$ (as well as $\psi_{\max}$)  sufficiently small. This \icr{allows large} values of \icr{for the} \ir{parameter $\gamma$ in Theorem \ref{thrmfreq}}. This means that the equilibrium phases $\delta^*$ can be farther apart and still belong to the  set  $\mathcal{C}(\gamma)$, \ir{as required}  in Theorem \ref{thrmfreq}. Thus this implies that the condition in Theorem \ref{thrmfreq}  would become \ir{even easier} to satisfy.}

	\end{rmk}
	

~\\
	\jo{In what follows, we
		\yy{show the} uniqueness of the phase differences for any equilibrium point  $\delta^*\in\mathcal{C}(\gamma)$.}

	\begin{prop}\label{propfreq}
			\jo{\yy{Consider} \lic{system~\eqref{slow2} and} the phase differences
				$\delta^*_{ij}\,\forall (i,j)\in\mathcal{E}$  at any equilibrium point  $\delta^*\in\mathcal{C}(\gamma)$  for some $\gamma\in[0,\pi/2-\psi_{\max})$, with $\psi_{\max}=\max_{(i,j)\in\mathcal{E}} \arctan(G_{ij}/B_{ij})$.	
				\yy{Then}}
			\lic{for each  $(i,j)\in\mathcal{E}$ the phase difference $\delta^\ast_{ij}$ has a unique value for all such equilibrium points.}	
		\end{prop}
		
		\ym{The proof can be found in Appendix \ref{proofsyn-b}.}

	\section {Simulation Results} \label{simureslt}
	To validate our theoretical analysis, we simulate
	\jo{an example associated with \eqref{slow2}. The network we consider}
	 has
	 $N$=$10$ \jo{oscillators} 
	 and $E$=$15$ 
	 coupling lines.
	 \yo{The 
	 	voltages \licc{$V_i, i\in\mathcal{V}$ are uniformly distributed in $[0.9,1]\,\text{pu}$,  which we denote as $V_i\in\text{uni}(0.9,1)\,\text{pu}$.}}
The \jo{coupling} parameters are
$G_{ij}\in\text{uni}(0.3,0.9)\,\text{S}, B_{ij}\in\text{uni}(1.2,2.9)\,\text{S}$,
 \mbox{$\forall (i,j)\in\mathcal{E}$}, $ \psi_{\max}=0.48\,\text{rad}$. 
\yo{The local loads are $P_{L_i}=\hat{G}_{ii}V_i^2~\,\forall i\in\mathcal{V}$ with  load conductances $\hat{G}_{ii} \in\text{uni}(0.02,0.05)\,\text{S}, \,\forall i\in\mathcal{V}$.}
The  gains \li{are} 
  $k_{pi}$=$0.1\,\text{Hz$/$pu}$ for $i=1$-$4,6,7$, $k_{pi}=\kkl{2\times 0.1}\,\text{Hz$/$pu}$  for $i=5,8$, and $k_{pi}=\kkl{3\times 0.1}\,\text{Hz$/$pu}$ for  $i=10$.
	We start the  simulation from  initial phases which  satisfy $\delta(0)\in\mathcal{C}(\gamma),\,\gamma<\pi/2-\psi_{\max}$.
\ym{Fig.~\ref{sychresult}(a) shows the phase angles and Fig.~\ref{sychresult}(b) the frequency deviations. 
	The phase angles converge to
the set $\mathcal{Y}_{\delta^*}$ in  $\mathcal{C}(\gamma)$ with $\|\mathcal{B}^{\top}\delta^*(t)\|_\infty=0.14\,\text{rad}<\gamma^\prime$,
$\gamma^\prime=\pi/2-\psi_{\max}=1.09 ~\text{rad}$.
Moreover, 	the frequency deviations $\dot\delta(t)$ converge to zero, which shows that the network  synchronizes.}
\yr{This is achieved with the condition in Theorem \ref{thrmfreq} satisfied, i.e $\delta^*\in\mathcal{C}(\gamma)$, $\gamma\in[0,\pi/2-\psi_{\max})$.
Moreover, the simulation shows that the nonlinear system is stable when the initial phases $\delta(0)$ also satisfy this condition.}

\yr{Furthermore, we compare  our result to \ir{the condition in \cite[Theorem $4.4$]{dorfler2012synchronization} that involves a bound by the} algebraic connectivity\footnote{The second smallest eigenvalue of a Laplacian.} \irr{$\lambda_2$}.
We compute    $\lambda_2(\mathcal{L}_y)=2.02$ where $\mathcal{L}_y=\mathcal{B}\text{diag}\left(\sqrt{a^2_{ij}+b^2_{ij}}\,\cos(\psi_{ij})\right)\mathcal{B}^{\top}$, and \ir{\irr{parameter}  $\lambda_{\text{critical}}$ in \cite[Theorem $4.4$]{dorfler2012synchronization} is computed} as $1.7\times10^4$. \ir{We have that} $\lambda_2(\mathcal{L}_y)$ is less than $\lambda_{\text{critical}}$, which shows that the condition in \cite[Theorem $4.4$]{dorfler2012synchronization}  fails even though the  system synchronizes. This is similarly observed for other network sizes.
Hence, the  condition in Theorem  \ref{thrmfreq}  is less conservative compared to that in \cite[Theorem $4.4$]{dorfler2012synchronization} \ir{when local asymptotic stability of the synchronization manifold is deduced}.}


	\section{Conclusion}\label{conclu}
\jo{We have considered in this paper the problem of frequency synchronization of interconnected oscillators in a Kuramoto-type model that takes coupling conductances  into account.
	We have provided a formal derivation using centre manifold theory
	\cj{of a stability result associated with such systems. In particular, if a certain condition}
	on the equilibrium point associated with the  coupling conductances and susceptances  \cj{is} satisfied, then the \cj{synchronization manifold} for the system considered is locally asymptotically stable.}
The results \licc{have been} demonstrated with simulations.

\begin{figure}[t]
	\includegraphics[width=.84\linewidth]{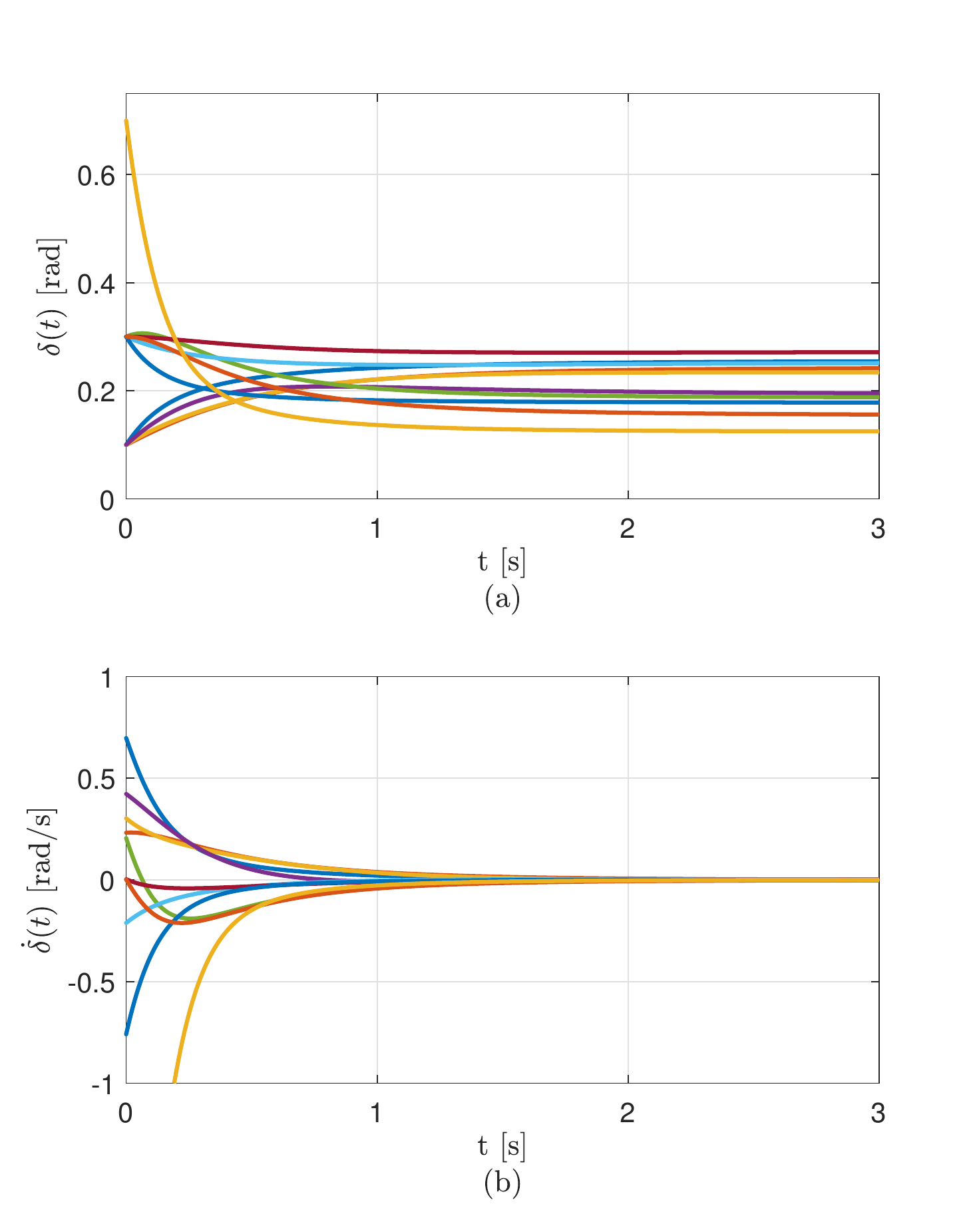}\vspace{-0.4cm}
	\caption{System responses: (a) phase angles; (b) frequency deviations}
	\label{sychresult}
\end{figure}

	\appendices

	\section{Proof of Theorem \ref{thrmfreq}} \label{proofsyn-a}	
\yr{\ir{We first note a property of
 the Jacobian associated with \eqref{slow2} which follows from Proposition \ref{lemL} in Appendix \ref{proofJacobian}. In particular,}
let system \eqref{slow2} be written
	as $\dot\delta_i=-k_{pi} H_i(\delta)$ ~$\forall i\in\,\mathcal{V}$, where
	\begin{equation}\label{H}
		H_i(\delta):= \varpi_i+ \mathsmaller{\sum}_{ j\in\mathcal{N}_i}(b_{ij}\sin(\delta_{ij})-a_{ij}\cos(\delta_{ij})).
	\end{equation}	
	Let the Jacobian  $\tilde{\mathcal{L}}(\delta^*)$ of the matrix of  $\begin{bmatrix} H_1(\delta) & \cdots & H_N(\delta) 	\end{bmatrix}^T$ be defined as follows\vspace{-0.125cm}	\begin{equation}\small \label{jacobianL}
			\tilde{\mathcal{L}} {(\delta^*)}=
			\begin{bmatrix} \small
				\tilde{\mathcal{L}}_{11} {(\delta^*)}  &\ldots& \tilde{\mathcal{L}}_{1N} {(\delta^*)} \\
				\vdots & \ddots & \vdots\\
				\tilde{\mathcal{L}}_{N1} {(\delta^*)}  &\ldots& \tilde{\mathcal{L}}_{NN} {(\delta^*)}
		\end{bmatrix}\vspace{-0.125cm}
	\end{equation}		
	where $ \tilde {\mathcal{L}}_{ij} {(\delta^*)}=	\dfrac{\partial  H_i(\delta)}{\partial\delta_j}\bigg|_{\delta=\delta^*}$,
	and $\delta^*\in\mathbb{T}^n$ is a fixed point.
	\ir{In Proposition} \ref{lemL} we show that $\tilde{\mathcal{L}}(\delta^*)$ is a non-symmetric Laplacian matrix 	if the condition on the equilibrium point in Theorem \ref{thrmfreq} is satisfied. \ir{Note that by Laplacian} matrix we mean that 	$\tilde{\mathcal{L}}_{ii}\kl{(\delta^*)}+ \mathsmaller{\sum}_{j\in\mathcal{N}_i}\tilde{\mathcal{L}}_{ij}\kl{(\delta^*)}=0, \, \forall i\in\mathcal{V}$ and    $\tilde{\mathcal{L}}_{ii}\kl{(\delta^*)}>0, \tilde{\mathcal{L}}_{ij}\kl{(\delta^*)}<0,\,\forall j\in\mathcal{N}_i$.}

\yr{We now proceed to prove Theorem~\ref{thrmfreq}. We} apply \ill{centre manifold theory as} described in \cite[Chapter $8$]{khalil2014}, 
\yl{and we invoke \lli{arguments similar} to those in \cite[Theorem $18.1.3$]{wiggins2003introduction} and  \cite[Theorem IV.$3$]{jouini2021frequency}}.
\yl{Let $\delta^*$ be a point generating the  manifold $\mathcal{Y}_{\delta^*}$ in~\eqref{synch-manifold}, and we denote by $\tilde \delta= \delta - \delta^*$  any deviation from $\delta^*$.}
 \lic{We rewrite} \eqref{slow2} \yl{as a shifted function as follows,}
\begin{equation}\label{delta-nonlinear1}
	\khl{\dot {\tilde \delta}}=f(\tilde \delta+ \delta^*)-f(\delta^*)\vspace{-0.125cm}
\end{equation}
where $f(\tilde \delta+ \delta^*):=\text{col}(\khl{-k_{pi} \tilde H_i(\tilde \delta+ \delta^*)})\in\mathbb{R}^{N}$,
$f(\delta^*):=\text{col}(\khl{-k_{pi} \tilde H_i( \delta^*)})\in\mathbb{R}^{N}$, and
\begin{equation}
\begin{split}
	\tilde H_i(\tilde \delta + \delta^*)&=  \mathsmaller{\sum}_{ j\in\mathcal{N}_i}(b_{ij}\sin(\tilde{\delta}_{ij}+\delta^*_{ij})-a_{ij}\cos(\tilde{\delta}_{ij}+\delta^*_{ij})),				\\
	\tilde H_i(\delta^*)&= \mathsmaller{\sum}_{ j\in\mathcal{N}_i}(b_{ij}\sin(\delta^*_{ij})-a_{ij}\cos(\delta^*_{ij})).	 \vspace{-0.125cm}
\end{split}	
\end{equation}
Model \eqref{delta-nonlinear1}  is then rewritten as
\begin{equation}\label{delta-nonlinear2}
	\dot{\tilde \delta}=-k_p\tilde{\mathcal{L}}{(\delta^*)} \tilde \delta + \underbrace{f(\tilde \delta+ \delta^*)-f(\delta^*)+k_p\tilde{\mathcal{L}}{(\delta^*)}\tilde\delta}_{\tilde{f}(\tilde \delta)}\vspace{-0.125cm}
\end{equation}
where  $\tilde{f}(0)=0, \frac{\partial \tilde{f}}{\partial \tilde \delta}(0)=0$. \ir{As noted above,}  by \yr{Proposition}~\ref{lemL} $\tilde{\mathcal{L}}{(\delta^*)}$ is a Laplacian matrix, {and let $k_p=\text{diag}(k_{pi})\in\mathbb{R}_{>0}^{N\times N}$.}
Noting that $k_p$ is  diagonal  with positive entries, the product \yl{$k_p\tilde{\mathcal{L}}{(\delta^*)}$} has only one zero eigenvalue and all others are positive. Hence there exists a transformation matrix $\yo{\mathcal{T}}\in\mathbb{R}^{N\times N}$ which is the eigenbasis of \yl{$k_p\tilde{\mathcal{L}}{(\delta^*)}$} such that
 $-\yo{\mathcal{T}}k_p\tilde{\mathcal{L}}{(\delta^*)}\mathcal{T}^{-1}=\begin{bmatrix}
	A_1 &0 \\0 &A_2
\end{bmatrix}$,
with $A_1=0$ being the single zero eigenvalue and the second block diagonal matrix $A_2\in\mathbb{R}^{(N-1)\times (N-1)}$ contains all  the remaining negative eigenvalues (i.e. $A_2$ is Hurwitz). Using the coordinate transformation \yl{$[\tilde y~\tilde z]^{\top}=\yo{\mathcal{T}}\tilde \delta$},
\yl{and noting that $[y ~z]^{\top}=\yo{\mathcal{T}} \delta$, $[y^*~ z^*]^{\top}=\yo{\mathcal{T}} \delta^*$, $\tilde y=y-y^*, \tilde z=z-z^*$,}
the \emph{normal form} of \eqref{delta-nonlinear2} is \khl{obtained~as} 
\begin{equation}\small\label{delta-nonlinear3}
	\begin{split}
	&\begin{bmatrix}
		\dot {\tilde y} \\ \dot {\tilde z}
	\end{bmatrix}=-\mathcal{T}k_p\tilde{\mathcal{L}}{\left(\mathcal{T}^{-1}\begin{bmatrix}	y^* \\ z^* \end{bmatrix}\right)} \mathcal{T}^{-1}\begin{bmatrix}
		\tilde y \\ \tilde z
	\end{bmatrix} + \mathcal{T}S(\tilde y  , \tilde z )\\
\end{split}
\end{equation}
where
\begin{equation}\small \label{S}
	\begin{split}
	S(\tilde y  , \tilde z ):=&
f\left(\mathcal{T}^{-1}\begin{bmatrix}	\tilde y+y^*  \\ \tilde z +z^*  \end{bmatrix}\right)
	-f\left( \mathcal{T}^{-1}\begin{bmatrix}	y^* \\ z^* \end{bmatrix}\right) \\
	& +k_p\tilde{\mathcal{L}}{\left(\mathcal{T}^{-1}\begin{bmatrix}	y^* \\ z^* \end{bmatrix}\right)}\mathcal{T}^{-1}
\begin{bmatrix}
	\tilde y \\ \tilde z
\end{bmatrix},\vspace{-0.125cm}
\end{split}
\end{equation}
which is compactly expressed as \vspace{-0.125cm}
\yl{\begin{subequations}\label{delta-nonlinear4}
	\begin{align}
		\label{delta-nonlinear4a}
	\dot {\tilde y}&=S_1(\tilde y , \tilde z)\vspace{-0.125cm} \\
	\label{delta-nonlinear4b}
	\dot {\tilde z}&=A_2\tilde z+S_2(\tilde y, \tilde z)
	\end{align}
\end{subequations}
where similarly to $\tilde{f}$ we must have
$S_i(0,0)=0,   \frac{\partial S_i}{\partial \tilde y}(0,0)=0, \frac{\partial S_i}{\partial \tilde z}(0,0)=0,\,\text{for} \,\, i=1,2$.
}
Since $A_1=0$ and $A_2$ is Hurwitz, then by \cite[Theorem $8.1$]{khalil2014}
there exists a  sufficiently small constant $\eta>0$ and a continuously differentiable~function \yl{$\chi(\tilde y)$}, defined for all \yl{$\|\tilde y\|<\eta$}, such that the invariant~manifold  \yl{$\tilde z=\chi(\tilde y)$} for \eqref{delta-nonlinear4} is \emph{\khl{a} \Nkl{centre manifold}} $W^c$ for \eqref{delta-nonlinear4}, that is, \vspace{-0.125cm}
\begin{equation}\label{center-manifold}
		\begin{split}
	W^c=&\left\lbrace (\tilde y,\tilde z)~\left| ~
	\begin{matrix}
	\tilde z=\chi(\tilde y),
	 \|\tilde y\|<\eta~,\\  \chi(0)=0, \frac{\partial \chi}{\partial \tilde y}(0)=0
	\end{matrix}\right.
\right\rbrace .\vspace{-0.125cm}
	\end{split}
\end{equation}
The   \emph{reduced system}  restricted to the \Nkl{centre manifold} $W^c$ is then given by \vspace{-0.125cm} \cite[Theorem $18.1.2$]{wiggins2003introduction}
\yl{\begin{equation}\label{reduced-manifold}
	\dot{\tilde{y}}=S_1(\tilde {y},\chi(\tilde {y})).\vspace{-0.125cm}
\end{equation}}
We now show that  \vspace{-0.125cm}
\begin{equation}\label{center-manifold2}
		W^c=\left\lbrace (\tilde y,\tilde z)\left| \hspace{-0.3cm}\begin{matrix}
			&\tilde y=y-y^*, \tilde z=z-z^*, \\& [ y~ z]^{\top}=\yo{\mathcal{T}}\delta, \, [	y^* ~ z^*]^{\top}  = \mathcal{T}  \delta^*,  \delta
			\in\mathcal{Y}_{\delta^*}
		\end{matrix} \right. \right\rbrace
	\end{equation}
is a \Nkl{centre manifold} for \eqref{delta-nonlinear4}.
To show this, note that  \eqref{center-manifold} implies that $W^c$ is  tangential to the \yl{$\tilde y$}-axis in  the neighbourhood 
of 
\yl{$\tilde y=0$.}
Hence we show that $W^c$ in \eqref{center-manifold2} has these properties.
\yl{First, $W^c$ in \eqref{center-manifold2} is invariant because it contains the steady states of \eqref{delta-nonlinear4}, which follows from the fact that $\dot {\tilde \delta}=\mathbf{0}_N$ since $\delta \in \mathcal{Y}_{\delta^*}$
}.
Second, $W^c$ in \eqref{center-manifold2} is tangential to the \yl{$\tilde y$-axis} at \yl{$\tilde y=0$}. To see this, 
for sufficiently small neighbourhood  of the origin
\yl{we want to satisfy\footnote{\yl{Note that we use $S(\tilde y, \tilde z)$ instead of $\chi (\tilde y)$ as we do not have the explicit expression of the latter. However, as $\chi (\tilde y)$ is a part of the steady-state solution of $\mathcal{T}S(\tilde y, \tilde z)$, the desired properties $ \chi(0)=0 $ and $\frac{\partial \chi}{\partial \tilde y}(0)=0$ can be ensured if $S(\tilde y, \tilde z)$ satisfies $ S(0,0)=0 $ and $\frac{\partial S}{\partial \tilde y}(0,0)=0$.}} $S(0,0)=0$ and $\frac{\partial S}{\partial \tilde y}(0,0)=0$. The condition $S(0,0)=0$ trivially holds from \eqref{S}. Using \eqref{S}, to ensure that  $\frac{\partial S}{\partial \tilde y}(0,0)=0$,}
  the Jacobian of \yl{$ f\left(\mathcal{T}^{-1}\begin{bmatrix}\tilde y+y^* \\ \tilde z+z^*\end{bmatrix} \right)-f\left(\yo{\mathcal{T}}^{-1}\begin{bmatrix}y^* \\ z^*\end{bmatrix}\right)=:\hat{S}(\tilde y, \tilde z)$} at 
\yl{$(\tilde y, \tilde z)=(0,0)$} defined by \vspace{-0.125cm}
	\begin{equation*} \small
		\begin{split}
			\frac{\partial \hat{S}(\tilde y ,\tilde z)}{\partial (\tilde y,\tilde z)}\bigg|_{(0,0)}
			&=\left. \begin{bmatrix}
				\frac{\partial \hat{S}_1(\tilde y,\tilde z)}{\partial \tilde y} &\frac{\partial \hat{S}_1(\tilde y,\tilde z)}{\partial \tilde z}\\
				\vdots & \vdots\\
				\frac{\partial \hat{S}_N (\tilde y,\tilde z)}{\partial \tilde y} &\frac{\partial \hat{S}_N(\tilde y,\tilde z)}{\partial \tilde  z}
			\end{bmatrix}\right|_{(0,0)}
		\end{split}
		\vspace{-0.125cm}
\end{equation*}
must be equal to $-k_p\tilde{\mathcal{L}}{\left(\mathcal{T}^{-1}\begin{bmatrix}	y^* \\ z^* \end{bmatrix}\right)}\yo{\mathcal{T}}^{-1}$.
We note that the columns of  the right eigenvector $\mathcal{T}^{-1}$ are given by
\yo{$\mathcal{T}^{-1}=[\text{span}(\mathbf{1}_N), v_2, \ldots, v_N]$ where $v_i \in \mathbb{R}^{n}$, $i=2, . . ., N$ are the column eigenvectors.}
Since \yl{$\tilde{\mathcal{L}}{\left(\mathcal{T}^{-1}\begin{bmatrix}	y^* \\ z^* \end{bmatrix}\right)}\text{span}(\mathbf{1}_N)=\mathbf{0}_N$}, then \yl{$-k_p\tilde{\mathcal{L}}{\left(\mathcal{T}^{-1}\begin{bmatrix}	y^* \\ z^* \end{bmatrix}\right)}\yo{\mathcal{T}}^{-1}$} has a zero first column. This means that the Jacobian entry corresponding to the $y$-component, i.e. \yl{$\left.\begin{bmatrix}
	 \frac{\partial \hat{S}_1(\tilde y, \tilde z)}{\partial \tilde y} \ldots \frac{\partial \hat{S}_N(\tilde y, \tilde z)}{\partial \tilde y}
\end{bmatrix}^{\top}\right|_{(0,0)}=\mathbf{0}_N$},
which shows that $W^c$ is tangent to the $\tilde y$-axis at 
\yl{the origin $(\tilde y, \tilde z)=(0,0)$}.
Therefore, 
for sufficiently small neighbourhood  of the origin
there exists a function \yl{$\tilde z=\chi(\tilde y)$} that satisfies the properties in  \eqref{center-manifold}, \illl{and $W^c$ is a centre manifold. Hence the solutions of \eqref{reduced-manifold} in a sufficiently small neighbourhood of the origin are in $W^c$ \cite[Theorem $18.1.2$]{wiggins2003introduction}. Since $W^c$ is composed of equilibrium points we have
\yl{$\dot{\tilde{y}}=0$},}
thus the reduced system \eqref{reduced-manifold} becomes \yl{$S_1(\tilde y,\chi(\tilde y))=0$}. This implies that \yl{$\tilde y_\chi(t)=\tilde y_\chi(0)\,\forall t$},
where \yl{$\tilde y_\chi(t)$} is the solution of \eqref{reduced-manifold},
\yl{hence \lli{the origin is stable in \eqref{reduced-manifold}}}.
\ill{By applying} 
\cite[Theorem $18.1.3$]{wiggins2003introduction}, 
the  \lli{origin in}  \eqref{reduced-manifold} being stable implies that
\yl{$\tilde y(t)=\tilde y_\chi(t)+w_1(t)$, $\tilde z(t)=\chi(\tilde y_\chi(t))+w_2(t)$ \illl{for some solution $\tilde y_\chi(t)$ in the centre manifold $W^c$, and some}
	\lli{$w_i(t)$ with  $\|w_i(t)\|<k_ie^{-\nkl{\rho_i} t},\,i=1,2,$} \illl{and} $k_i,\nkl{\rho_i}>0$.
Consequently,
$\lim_{t\to\infty}(\tilde y(t),\tilde z(t))=(\tilde y_\chi(0),\chi(\tilde y_\chi(0)))\illl{\in W^c}$ and 
$\lim_{t\to\infty}\tilde \delta(t)=\yo{\mathcal{T}}^{-1}\nkl{[\tilde y_\chi(0)~~\chi(\tilde y_\chi(0))]^\top}$.}

\licc{Therefore all solutions $\delta$ starting  in a sufficiently small  neighbourhood of $\delta^*$ satisfy $\lim_{t\to\infty}\|\delta(t)-\delta^\dagger\|\to0$
	for some $\delta^\dagger\in\mathcal{Y}_{\delta^*}$. The same also holds for solutions starting in a sufficiently small neighbourhood of any point in $\mathcal{Y}_{\delta^*}$, \lli{which follows from the same arguments as the ones that have been used for solutions with initial conditions in a neighbourhood of~$\delta^\ast$}.  Therefore all solutions $\delta$ starting  in a sufficiently small  neighbourhood $\Xi$ of $\mathcal{Y}_{\delta^*}$ satisfy $\lim_{t\to\infty}\|\delta(t)-\delta^\dagger\|\to0$
	for some $\delta^\dagger\in\mathcal{Y}_{\delta^*}$.}

\licc{From the boundedness of trajectories starting in $\Xi$, which can be chosen to be compact, we can also deduce the stability of~$\mathcal{Y}_{\delta^*}$. This hence also implies its asymptotic stability using the convergence property mentioned in the previous paragraph.}
~\hfill\mbox{$\hfill\blacksquare$}

\section{Proof of Proposition \ref{propfreq}} \label{proofsyn-b}


		\yy{Since}  $\sin(\cdot)$ and $\cos(\cdot)$ are real analytic, the right-hand side of \eqref{slow2} is real analytic. For $f_{\kl{i}}:\mathcal{C}(\gamma) \to\mathbb{R}$, let the right-hand side of \eqref{slow2} be $f_{\kl{i}}(\delta)$.
		Then the right-hand side of \eqref{slow2} is injective in the domain $\mathcal{C}(\gamma), \gamma\in[0, \pi/2-\psi_{\max})$,
		that is, $f_{\kl{i}}(\kl{\kappa}^*_1)=f_{\kl{i}}(\kl{\kappa}^*_2)$ if and only if $\kl{\kappa}^*_1=\kl{\kappa}^*_2$. Thus,  $f_{\kl{i}}(\kl{\delta})$  is a one-to-one function for $\kl{\delta}^*\in\mathcal{C}(\gamma), \gamma\in[0, \pi/2-\psi_{\max})$
		 and the uniqueness of the \yl{angle differences in 
		 	 $\mathcal{C}{(\gamma)}$} follows.
		 $\hfill\blacksquare$


\ir{
\section{Proposition \ref{lemL} and its proof}	 	\label{proofJacobian}
}
\begin{prop}[System expanded Laplacian matrix]\label{lemL}
		\yr{Consider system \eqref{slow2} together with \eqref{jacobianL}. For a fixed point $\delta^*\in\mathbb{T}^n$,	
		the Jacobian $\tilde{\mathcal{L}}{(\delta^*)}$ \ir{in \eqref{jacobianL}} is given by}\vspace{-0.125cm}
\begin{equation}
			\label{laplace1}
\hspace{-15mm}\tilde{\mathcal{L}}{(\delta^*)}=\mathcal{B}\text{diag}\left(\sqrt{a^2_{ij}+b^2_{ij}}\,\cos(\delta^*_{ij}-\psi_{ij})\right)\mathcal{B}^{\top}.\vspace{-0.125cm}
		\end{equation}
		 which can be  rewritten   as \vspace{-0.125cm}
\begin{equation} \hspace{-1mm}
			\label{laplace}
			\tilde{\mathcal{L}}{(\delta^*)}=\mathcal{B}\text{diag}(b_{ij}\cos(\delta^*_{ij}))\mathcal{B}^{\top}\hspace{-1mm}+ \mathbf{\hat{B}}\text{diag}(a_{ij}\sin(\delta^*_{ij}))\mathcal{B}^{\top} \vspace{-0.125cm}
		\end{equation}
	
		Moreover, if $\delta^\ast$ \yr{is an equilibrium point \ir{that}} satisfies
		$\delta^*\in\mathcal{C}(\gamma)$ \ill{for some} $\gamma\in[0,\pi/2-\psi_{\max})$ \jo{with $\psi_{\max}=\max_{(i,j)\in\mathcal{E}} \arctan(G_{ij}/B_{ij})$}, then
		$\tilde{\mathcal{L}}{(\delta^*)}$ is a {non-symmetric\footnote{\yr{\ir{Note that form} \eqref{laplace} \ym{\ir{illustrates} the structure of} \eqref{laplace1} where   $\mathcal{B}\text{diag}(b_{ij}\cos(\delta^*_{ij}))\mathcal{B}^{\top}$ is the lossless symmetric part, while the lossy non-symmetric part is $\mathbf{\hat{B}}\text{diag}(a_{ij}\sin(\delta^*_{ij}))\mathcal{B}^{\top}$.
			For  lossless networks, $\tilde{\mathcal{L}}\kl{(\delta^*)}$ reduces to the symmetric part since $a_{ij}=0$ ($G_{ij}=0$, $\psi_{ij}=0$).}}} Laplacian
		{matrix.}
	\end{prop}
\begin{proof}
		The proof makes use of  Lemma \ref{altmodel} in Appendix~\ref{proofaltmodel}. Lemma \ref{altmodel}  shows   that \eqref{slow2} is equivalent to
			\begin{equation} \hspace{-0.4cm}\label{slowsine1}
			\dot{\delta}_i=k_{pi}\varpi_i- k_{pi} \hspace{-0.1cm}\sum_{ j\in\mathcal{N}_i}
			\sqrt{a^2_{ij}+b^2_{ij}}\,\sin(\delta_{ij}-\psi_{ij}), \hspace{2mm} \forall i\in\mathcal{V} \vspace{-0.125cm}	\end{equation}
		where \mbox{$\psi_{ij}=\arctan(a_{ij}/b_{ij})=\arctan(G_{ij}/B_{ij})$}.
		\yr{

\ir{We now write} \eqref{slow2}, \eqref{slowsine1}}
as $\dot\delta_i=-k_{pi} H_i(\delta)$ 
\ir{as in \eqref{H} where}
\begin{align}
	H_i(\delta)&:= \varpi_i+ \mathsmaller{\sum}_{ j\in\mathcal{N}_i}(b_{ij}\sin(\delta_{ij})-a_{ij}\cos(\delta_{ij})),				\label{slow6}\\
	&=\varpi_i+  \mathsmaller{\sum}_{ j\in\mathcal{N}_i}\sqrt{a^2_{ij}+b^2_{ij}}\,\sin(\delta_{ij}-\psi_{ij}).\label{slowsine2}\vspace{-0.125cm}
\end{align}	
Considering \eqref{slow6},
we have the diagonal and off-diagonal entries of $\tilde{\mathcal{L}}{(\delta^*)}$ as\vspace{-0.125cm}		
\yo{\begin{equation}\label{hd2}				
	\begin{split}
		&\tilde{\mathcal{L}}_{ii} {(\delta^*)}=	 \mathsmaller{\sum}_{j\in\mathcal{N}_i}b_{ij}\cos(\delta^*_{ij})+ \mathsmaller{\sum}_{ j\in\mathcal{N}_i}a_{ij}\sin(\delta^*_{ij})\vspace{-0.125cm}\\
		&\tilde{\mathcal{L}}_{ij} {(\delta^*)}=-b_{ij}\cos(\delta^*_{ij})-a_{ij}\sin(\delta^*_{ij})~~~~~~~~ j\in\mathcal{N}_i.
	\end{split}\vspace{-0.125cm}
\end{equation}			
Similarly, we obtain by considering~\eqref{slowsine2}	\vspace{-0.125cm}	
\begin{equation}\label{hd1}				
	\begin{split}
		&\tilde{\mathcal{L}}_{ii} {(\delta^*)}=	 \mathsmaller{\sum}_{ j\in\mathcal{N}_i}\sqrt{a^2_{ij}+b^2_{ij}}\,\cos(\delta^*_{ij}-\psi_{ij}) 
\\
		&\tilde{\mathcal{L}}_{ij} {(\delta^*)}=-\sqrt{a^2_{ij}+b^2_{ij}}\,\cos(\delta^*_{ij}-\psi_{ij})~~~~~~ j\in\mathcal{N}_i.
	\end{split}\vspace{-0.125cm}
\end{equation}}\yo{Noting that} $a_{ji}=a_{ij}$, $b_{ji}=b_{ij}$,
$\cos(\delta^*_{ji})=\cos(\delta^*_{ij})$, $\sin(\delta^*_{ji})=-\sin(\delta^*_{ij})$, \ill{using} 
\eqref{hd2} for all the entries of \eqref{jacobianL} gives the form~\eqref{laplace}. Likewise, \ill{using} 
\eqref{hd1} for all the entries of~\eqref{jacobianL} and separating the  $\cos(\cdot)$ and $\sin(\cdot)$ terms gives the   form~\eqref{laplace1}. 	
This completes the first part.

Note that 	$\tilde{\mathcal{L}}_{ii}\kl{(\delta^*)}+ \mathsmaller{\sum}_{j\in\mathcal{N}_i}\tilde{\mathcal{L}}_{ij}\kl{(\delta^*)}=0, \, \forall i\in\mathcal{V}$ which shows that 	$\tilde{\mathcal{L}}\kl{(\delta^*)}\mathbf{1}_N=\mathbf{0}_N$.
Consider now
$\delta^*\in\mathcal{C}(\gamma)$, $\gamma\in[0,\pi/2-\psi_{\max})$; \ill{then} we have   $\tilde{\mathcal{L}}_{ii}\kl{(\delta^*)}>0,\,\forall i\in\mathcal{V}$ and  $\tilde{\mathcal{L}}_{ij}\kl{(\delta^*)}<0,\,\forall j\in\mathcal{N}_i$.
Furthermore,
\lic{we have that $\tilde{\mathcal{L}}_{ii}\kl{(\delta^*)}=- \mathsmaller{\sum}_{j\in\mathcal{N}_i}\tilde{\mathcal{L}}_{ij}\kl{(\delta^*)}$ hence $\tilde{\mathcal{L}}\kl{(\delta^*)}$ is a    Laplacian.}
The fact that $\tilde{\mathcal{L}}{(\delta^*)}^{\top}\neq \tilde{\mathcal{L}}{(\delta^*)}$ since $\hat {\mathbf{B}}^{\top}\neq \mathcal{B}^{\top}$	 shows that $\tilde{\mathcal{L}}{(\delta^*)}$ is   a non-symmetric Laplacian~matrix.	
\end{proof}

\yr{
\section{Lemma \ref{altmodel} and its proof}	 	\label{proofaltmodel}}
\khlb{\begin{lem}[Sine representation]
	\label{altmodel}
	System \eqref{slow2} \li{has the following equivalent representation}
	\begin{equation*} \hspace{-0.4cm} 
		\dot{\delta}_i=k_{pi}\varpi_i- k_{pi} \hspace{-0.1cm}\sum_{ j\in\mathcal{N}_i}
		\sqrt{a^2_{ij}+b^2_{ij}}\,\sin(\delta_{ij}-\psi_{ij}), \hspace{2mm} \forall i\in\mathcal{V} 
\end{equation*}
	where \mbox{$\psi_{ij}=\arctan(a_{ij}/b_{ij})=\arctan(G_{ij}/B_{ij})$}.
\end{lem}}
\begin{proof}
Considering  \eqref{slow2}, we rewrite
\mbox{$P_{ij}\lli{:=}b_{ij}\sin(\delta_{ij})-a_{ij}\cos(\delta_{ij})$} as ${\scriptsize P_{ij}= \left(\frac{b_{ij}\sin(\delta_{ij})}{\sqrt{a^2_{ij}+b^2_{ij}}}-\frac{a_{ij}\cos(\delta_{ij})}{\sqrt{a^2_{ij}+b^2_{ij}}}\right)\sqrt{a^2_{ij}+b^2_{ij}}}$.
Let $\scriptsize{\cos(\psi_{ij})=\frac{b_{ij}}{\sqrt{a^2_{ij}+b^2_{ij}}}}$, $\scriptsize{\sin(\psi_{ij})}=\frac{a_{ij}}{\sqrt{b^2_{ij}+b^2_{ij}}}$, thus $\scriptsize{\psi_{ij}=\arctan(\frac{\sin(\psi_{ij})}{\cos(\psi_{ij})})}=\arctan(G_{ij}/B_{ij})$.
Then ${\scriptsize P_{ij}=\sqrt{a^2_{ij}+b^2_{ij}}\,\left(\cos(\psi_{ij})\sin(\delta_{ij})-\sin(\psi_{ij})\cos(\delta_{ij})\right)}$~or 	${\scriptsize P_{ij}=}$${\sqrt{a^2_{ij}+b^2_{ij}}\,\sin(\delta_{ij}-\psi_{ij})}$
which \lli{transforms} 
\eqref{slow2} \lli{into}~$\dot{\delta}_i$ of Lemma~\ref{altmodel}. \end{proof}

\bibliographystyle{IEEEtran}
\bibliography{bib_inverterbased}

\end{document}